\documentclass[final,leqno,letterpaper]{etna_no_ref}

\setbibdata{1}{xx}{xx}{xxxx} %First page/Last page/volume/year

\usepackage[english]{babel}
\usepackage{graphicx}
\usepackage{amssymb}
\usepackage{amsmath}
\usepackage{latexsym}
\usepackage{eucal}                % allow use of Euler script fonts in math mode
\usepackage{setspace}
\usepackage[dvipsnames]{xcolor}
\usepackage{comment}
\usepackage{url}
\usepackage[space,noadjust,nocompress]{cite}

\usepackage{listings}
\usepackage{algorithm,algorithmicx,algpseudocode}
\floatname{algorithm}{Algorithm} 

\newtheorem{assumption}[theorem]{Assumption}

\newcommand{\N}{\mathbb{N}}

\renewcommand{\Re}{\mathbb{R}} 

\renewcommand{\hat}{\widehat}

\def\bb{\mathbf b}

\def\bd{\mathbf d}

\def\bg{\mathbf g}

\def\bp{\mathbf p}

\def\br{\mathbf r}
\def\bs{\mathbf s}
\def\bu{\mathbf u}
\def\bv{\mathbf v}
\def\bu{\mathbf u}
\def\bx{\mathbf x}
\def\by{\mathbf y}
\def\bz{\mathbf z}
\def\bzero{\mathbf 0}
\def\bone{\mathbf 1}
\def\bbeta{\boldsymbol{\beta}}

\def\bmu{\boldsymbol{\mu}}
\def\bnu{\boldsymbol{\nu}}

\def\bphi{\boldsymbol{\varphi}}
\def\brho{\boldsymbol{\rho}}
\def\mA{\mathcal{A}}

\def\mD{\mathcal{D}}
\def\mF{\mathcal{F}}

\def\mL{\mathcal{L}}

\def\mT{\mathcal{T}}

\usepackage{scalerel}

\DeclareMathOperator*{\argmin}{argmin}

\DeclareMathOperator*{\sign}{sign}
\DeclareMathOperator*{\pr}{proj}
\def\st{\mbox{s.t.}}
\def\diag{\mbox{diag}}

\title{A subspace-accelerated split Bregman method \\ for sparse data recovery with joint $\ell_1$-type regularizers \thanks{This work was partially supported by Gruppo Nazionale per il Calcolo Scientifico - Istituto Nazionale di Alta Matematica (GNCS-INdAM) and by the VAIN-HOPES Project, funded by the 2019 V:ALERE (VAnviteLli pEr la RicErca) Program of the University of Campania ``L. Vanvitelli''.}}

\author{
Valentina De Simone\footnotemark[2] \and Daniela di Serafino\footnotemark[2] \and Marco Viola\footnotemark[2]}

\shorttitle{A subspace-accelerated split Bregman method} 
\shortauthor{V.~DE~SIMONE., D.~DI SERAFINO, AND M.~VIOLA}

\date{March 23, 2020}

\begin{document}
	
	\maketitle
    \centerline{\footnotesize VERSION 2 -- Mar 23, 2020}
    	
	\renewcommand{\thefootnote}{\fnsymbol{footnote}}
	
	\footnotetext[2]{Department of Mathematics and Physics, University of Campania ``L. Vanvitelli'', viale A. Lincoln, 5, Caserta (Italy), \texttt{\{valentina.desimone, daniela.diserafino, marco.viola\}@unicampania.it}}

	\begin{abstract}
		We propose a subspace-accelerated Bregman method for the linearly constrained minimization of functions of the form
		$ f(\bu) + \tau_1\,\|\bu\|_1 + \tau_2\,\|D\,\bu\|_1 $, where $f$ is a smooth convex function and $D$ represents a linear operator,
		e.g.~a finite difference operator, as in anisotropic Total Variation and fused-lasso regularizations. Problems of this type arise in a wide
		variety of applications, including portfolio optimization, learning of predictive models from functional Magnetic Resonance Imaging (fMRI) data, and source detection problems in electroencephalography. The use of $\|D\,\bu\|_1$
		is aimed at encouraging structured sparsity in the solution. The subspaces where the acceleration is performed are selected so that the restriction
		of the objective function is a smooth function in a neighborhood of the current iterate. Numerical experiments on multi-period portfolio
		selection problems using real datasets show the effectiveness of the proposed method.
	\end{abstract}
	
	
	
\section{Introduction\label{sec:intro}}
	
We are interested in the solution of problems of the form
\begin{equation} \label{TVl1reg_problem}
	\begin{array}{rl}
	\min & \displaystyle f(\bu) + \tau_1\,\|\bu\|_1 + \tau_2\, \|D\,\bu\|_1\\
	\st & \displaystyle A\bu = \bb,
	\end{array}
\end{equation}
where $ f:\Re^n\rightarrow\Re$ is a closed convex function at least twice continuously differentiable, $\bu\in\Re^n$, $ D\in\Re^{q\times n} $, $ A\in\Re^{m\times n} $, and
$ \bb\in\Re^m $.
The $\ell_1$ regularization term in the objective function encourages sparsity in the solution while the use of $\|D\,\bu\|_1$ is aimed at incorporating further information about the solution.  For example, in the case of discrete anisoptropic Total Variation \cite{Esedoglu:2004,Goldstein:2009SB}, $D$ is a first-order finite-difference operator and the regularization encourages smoothness along certain directions. The combination of the two regularization terms can be seen as a generalization of the fused lasso regularization introduced in~\cite{Tibshirani:2005} in the case of least-squares regression. Problems of the form \eqref{TVl1reg_problem} arise, e.g., in multi-period portfolio optimization~\cite{Corsaro:2019anorFused}, in predictive modeling and classification (machine learning) on functional Magnetic Resonance Imaging (fMRI) data~\cite{Dohmatob:2014,Baldassarre:2012}, in source detection problems in electroencephalography~\cite{Becker:2017}, and in multiple change-point detection~\cite{Niu:2016}.

%The nonsmoothness of the $\ell_1$-type regularization terms, together with the presence of linear constraints, makes the solution of problem \eqref{TVl1reg_problem} challenging. 
Methods based on Bregman iterations \cite{Bregman:1967,Osher:2005mms,Goldstein:2009SB,Campagna:2017} have proved to be efficient in the solution of this type of problems. As we will see in Section~\ref{sec:sb}, the Bregman iterative scheme requires at each step the solution of an $\ell_1$-regularized unconstrained optimization subproblem. For this minimization, which does not need to be performed exactly, but generally requires high accuracy (see Theorem~\ref{theorem:Bregman_Convergence_Eckstein} in Section~\ref{sec:sb}), one can use iterative methods suited to deal with the $\ell_1$-regularization term, such as FISTA \cite{Beck:2009FISTA}, SpaRSA \cite{Wright:2009SpaRSA}, BOSVS \cite{Chen:2013BOSVS} and ADMM~\cite{Boyd:2011admm}.
%By mean of slack variables into the model this can indeed lead to replacing the nonsmooth subproblems with smooth ones, combined the application of the so-called \textit{soft thresholding operator}.
A possible drawback is that these methods may be inefficient when high accuracy is required.

Herein, we propose a subspace-acceleration strategy for the Bregman iterative scheme, which is aimed at replacing, at certain steps, the unconstrained minimization of the $\ell_1$-regularized subproblem with the unconstrained minimization of a smooth restriction of it to a suitable subspace. The proposed strategy finds its roots in the class of orthant-based methods \cite{Andrew:2007owl,Byrd:2012MathProg,Keskar:2016OBA} for $\ell_1$-regularized minimization problems, which are based on the consecutive minimization of smooth approximations to the problem over a sequence of orthants. However, by following \cite{Robinson:2017FaRSA}, instead of considering the restriction to the full orthant, we restrict the minimization to the orthant face identified by the zero variables. Ideally, one would like to perform this subspace minimization only when there is guarantee that the subproblem solution will lie on that orthant face.  However, this is unpractical to check. For this reason, starting from the work in \cite{Robinson:2017FaRSA}, we introduce a switching criterion to decide whether to perform the subspace-acceleration step. The criterion is based on the use of some optimality measures for the current iterate with respect to the current subproblem. More specifically, it is based on a comparison between a measure of the optimality violation of the zero variables and a measure of the optimality violation of the other variables. This strategy comes from the adaptation to $\ell_1$-regularized optimization of the concept of proportional iterates, developed in the case of quadratic optimization problems subject to bound constraints or to bound constraints and a single linear equality constraint \cite{diSerafinoToraldoViolaBarlow:2018,Dostal:1997,Dostal:2005,Friedlander:1989,Friedlander:1994b,Robinson:2015}.

The idea of introducing acceleration steps over suitable subspaces to improve the performance of splitting methods for problem~\eqref{TVl1reg_problem} is not new. An example is provided, e.g., by~\cite{Chen:2011jcam}. However, the strategy we propose in this work differs from that subspace-acceleration strategy because we focus on Bregman iterations and aim at replacing nonsmooth unconstrained subproblems with smooth smaller ones, while the algorithm in~\cite{Chen:2011jcam} is based on the introduction of a subspace-acceleration step after the minimization steps in an ADMM algorithm \cite{Boyd:2011admm}, where the subspace is spanned by directions obtained by using information from previous iterations.

This paper is organized as follows. In Section~\ref{sec:preliminaries} we recall some convex analysis concepts that will be used later in this work. In Section~\ref{sec:sb} we briefly describe the Bregman iterative scheme for the solution of problem~\eqref{TVl1reg_problem} and prove its convergence in the case of inexact subproblem minimization. In Section~\ref{sec:acceleration} we show how suitable subspace-acceleration steps can be introduced into the Bregman iterative scheme. In Section~\ref{sec:portfolio} we report numerical results for the solution of portfolio optimization problems modeled by~\eqref{TVl1reg_problem}. We provide some conclusions in Section~\ref{sec:conclusions}.

\subsection*{Notation\label{sec:notation}}
Throughout this paper scalars are denoted by lightface Roman or Greek fonts, e.g., $a, \alpha \in \Re$, vectors by boldface Roman or Greek fonts, e.g., $\bv, \bmu \in \Re^n$. The $i$-th entry of a vector $\bv \in \Re^n$ is denoted $v_i$ or $[v]_i$. Given a continuously differentiable function $F(\bx):\Re^n\rightarrow\Re$, we use $\nabla_i F(\bx)$ to indicate the first derivative of $F$ with respect to the variable $x_i$. We use $\bzero_n$ and $\bone_n$ to indicate the vectors in $\Re^n$ with all entries equal to $0$ and $1$, respectively; the subscript is omitted if the dimension is clear from the context. For any vectors $ \bu\in\Re^{n_1} $ and $ \bv\in\Re^{n_2} $ we use the notation $ [\bu;\,\bv] $ to represent the vector $ [\bu^\top,\,\bv^\top]^\top \in \Re^{n_1+n_2} $. The Euclidean scalar product between $ \bu,\;\bv\in\Re^n $ is indicated as $ \langle\bu,\bv\rangle $. Norms $\| \cdot \|$ are $\ell_2$. Superscripts are used to denote the elements of a sequence, e.g., $\left\lbrace \bx^k\right\rbrace $.
	
\section{Preliminaries\label{sec:preliminaries}}
We recall some concepts that will be used in the next sections.

\begin{definition}
Given a function $ Q:\Re^n\rightarrow\Re $, the convex conjugate $Q^*$ of $Q$ is defined as follows:
$$  Q^*(\by) = \sup\limits_{\bx} \left\langle \by,\bx\right\rangle  - Q(\bx).$$
\end{definition}
Note that $ Q^* $ is a closed convex function for any given $ Q $.  If $Q$ is strictly convex, then $ Q^* $ is also continuously differentiable;
moreover, if $Q$ is a closed convex function, then $ Q^{**}(\bx) = Q(\bx) $ \cite{HiriartUrruty:1996II}.

\begin{definition}
Given a closed convex function $ Q:\Re^n\rightarrow \Re $, a vector $ \bp\in\Re^n $ is said a subgradient of $Q$ at a point $ \bx \in \Re^n $ if
$$  Q(\by)-Q(\bx) \geq \left\langle\bp,\,\by-\bx\right\rangle,\quad \forall \, \by \in \Re^n. $$
The set of all the subgradients of $ Q $ at $ \bx $ is referred to as the subdifferential of $Q$ at $\bx$, and is denoted $ \partial Q(\bx) $.
\end{definition}

If $Q$ is a closed convex function, then \cite[Chapter X]{HiriartUrruty:1996II}
\begin{equation}\label{conjugate_subdifferential}
\bp\in\partial Q(\bx) \quad \mbox{if and only if} \quad \bx\in\partial Q^*(\bp).
\end{equation}
Moreover, we have that $ Q(\bx)+Q^*(\bp) = \langle \bp,\,\bx\rangle $.

\begin{definition}
A point-to-set map $ \Phi:\Re^n\rightarrow 2^{\Re^n} $ is said to be a monotone operator if
$$ \left\langle \bx-\by,\,\bu-\bv\right\rangle  \geq 0,\qquad\mbox{for all}\;\; \bx,\by\in\Re^n,\; \bu\in\Phi(\bx),\; \bv\in\Phi(\by).$$
Moreover, $ \Phi $ is said to be maximal monotone if it is monotone and its graph, i.e., the set 
$$\left\lbrace (\bx,\by) \in\Re^n\times\Re^n \;:\; \by \in \Psi(\bx) \right\rbrace,$$
is not strictly contained in the graph of any other monotone operator.
\end{definition}

An example of maximal monotone operator is the subdifferential of a lower-semicontinuous convex function (see \cite{Rockafellar:1976} and references therein).

\begin{definition}
Given an operator $\Phi:\Re^n\rightarrow 2^{\Re^n}$, the inverse of $\Phi$ is the operator $\Phi^{-1}:\Re^n\rightarrow 2^{\Re^n}$ defined as
$$ \Phi^{-1}(\by) = \left\lbrace \bx\in\Re^n \;:\; \by\in\Phi(\bx)\right\rbrace. $$
\end{definition}

%\begin{definition}
%Given an operator $ \Phi $ and a scalar $ c>0 $, the operator $ J_{c\Phi} = (I+c\Phi)^{-1}$ is called a resolvent of $ \Phi $.
%\end{definition}

\section{The split Bregman method\label{sec:sb}}
For the sake of simplicity and self consistency, we briefly describe the split Bregman method \cite{Goldstein:2009SB} for the solution of $\ell_1$-regularized problems of type~\eqref{TVl1reg_problem}. In order to separate the two $\ell_1$-regularization terms, we introduce the auxiliary variable $ \bd = D\bu $, so that problem~\eqref{TVl1reg_problem} can be reformulated as
\begin{equation} \label{const_l1reg_problem}
	\begin{array}{rl}
	\min & \displaystyle E(\bu,\bd) \equiv  f(\bu) + \tau_1\,\|\bu\|_1 + \tau_2\,\|\bd\|_1\\
	\mbox{s.t.} & \displaystyle A\bu \qquad = \bb,\\
		        & \displaystyle D\,\bu - \bd = \bzero.
	\end{array}
\end{equation}

The split Bregman method is based on a Bregman iterative scheme for the solution of~\eqref{const_l1reg_problem}. Letting $ \bu^0\in\Re^n$, $\bd^0\in\Re^q$,
and $\bp^0 = \left[ \bp^0_\bu ; \, \bp^0_\bd \right] \in \partial E(\bu^0,\bd^0)$, the $k$-th iteration of the Bregman method reads as follows:
\begin{eqnarray}
	\qquad\left[\bu^{k+1};\,\bd^{k+1}\right] & = &
	\argmin\limits_{\bu,\bd} \mD^{\bp^k}_E\left(\left[\bu;\,\bd\right],\, \left[\bu^k; \bd^k\right] \right) + \dfrac{\lambda}{2}\Vert A\,\bu-\bb\Vert^2 + \dfrac{\lambda}{2}\Vert D\,\bu-\bd\Vert^2,\\
	\bp^{k+1}_\bu &=& \bp^k_\bu - \lambda A^\top (A\bu^{k+1}-\bb) - \lambda D^\top (D\bu^{k+1} - \bd^{k+1}),\\[3pt]
	\bp^{k+1}_\bd &=& \bp^k_\bd - \lambda (\bd^{k+1} - D\bu^{k+1}),
\end{eqnarray}
where $ \bp^k = \left[ \bp^k_\bu ; \, \bp^k_\bd \right] $ and
\begin{equation*}
\mD^{\bar{\bp}}_E\left([\bu ; \, \bd] ,\, [\bar{\bu} ; \, \bar{\bd}]\right) = E(\bu,\bd) - E(\bar{\bu},\bar{\bd}) - \left\langle\bar{\bp}_\bu,\,\bu-\bar{\bu}\right\rangle - \left\langle\bar{\bp}_\bd,\,\bd-\bar{\bd}\right\rangle,
\end{equation*}
with $ \bar{\bp} \in \partial E(\bar{\bu},\bar{\bd})$, is the so-called Bregman distance associated with the convex function $ E $ at the point $ [ \bar{\bu} ; \, \bar{\bd}] $.
	
Following \cite{Osher:2005mms,Goldstein:2009SB}, thanks to the linearity of the equality constraints, a simplified iteration can be used in place of the original Bregman one:
\begin{eqnarray}
	\left[\bu^{k+1} ; \, \bd^{k+1}\right] & = & \argmin\limits_{\bu,\bd} E(\bu,\bd) + \dfrac{\lambda}{2}\Vert A\,\bu-\bb^k_\bu\Vert^2 + \dfrac{\lambda}{2}\Vert D\,\bu-\bd-\bb^k_\bd\Vert^2,\label{SimpBreg_UD_1}\\
	\bb^{k+1}_\bu &=& \bb^k_\bu +\bb - A\bu^{k+1},\label{SimpBreg_UD_2}\\[3pt]
	\bb^{k+1}_\bd &=& \bb^k_\bd +\bd^{k+1} - D\bu^{k+1}.\label{SimpBreg_UD_3}
\end{eqnarray}
	
In order to simplify the notation, it is convenient to rewrite \eqref{const_l1reg_problem} in terms of a single variable $\bx$ as
\begin{equation}
\begin{array}{rl}
	\min           & \displaystyle K(\bx) \equiv F(\bx) + \sum_{i=1}^{n+q}\delta_i |x_i|\\
	\st & M\,\bx = \bs,
\end{array} \label{const_l1reg_prob_unk_X}
\end{equation}
where
\begin{equation}
     \bx= \left[\begin{array}{c}\bu\\\bd\end{array}\right],\quad F(\bx) = f(\bu),\quad
     M = \left[\begin{array}{cc}A&0\\D&-I\end{array}\right], \quad \bs= \left[\begin{array}{c}\bb\\\bzero\end{array}\right], \label{const_l1reg_prob_X_definitions}
\end{equation}
and
$$
     \delta_i = \left\lbrace\begin{array}{ll}\tau_1,&\mbox{if }i\leq n,\\\tau_2,&\mbox{if }i>n.\end{array} \right.
$$
We also denote $ n_x=n+q $ the size of $ \bx $ and $ n_s = m+q $ the number of rows of $ M $ (i.e., the size of $\bs$), so that $ M\in\Re^{n_s\times n_x} $.
Then, iteration \eqref{SimpBreg_UD_1}-\eqref{SimpBreg_UD_3} can be written as
\begin{eqnarray}
	\bx^{k+1} & = & \argmin\limits_{\bx} \, K(\bx) + \dfrac{\lambda}{2}\Vert M\,\bx-\bs^k\Vert^2, \label{SimpBreg_X_1}\\
	\bs^{k+1} & = & \bs^k +\bs - M\,\bx^{k+1}, \label{SimpBreg_X_2}
\end{eqnarray}
where $ \bs^k= \left[\bb^k_\bu;\,\bb^k_\bd\right]$.
	
We can rewrite \eqref{SimpBreg_X_1}-\eqref{SimpBreg_X_2} as the augmented Lagrangian iteration
\begin{eqnarray}
	\bx^{k+1} & = & \argmin\limits_{\bx} \, K(\bx)- \left\langle\bmu^k,\,M\bx\right\rangle + \dfrac{\lambda}{2}\Vert M\,\bx-\bs\Vert^2, \label{AugLag_X_1}\\
	\bmu^{k+1} & = & \bmu^k + \lambda\left(\bs - M\,\bx^{k+1}\right), \label{AugLag_X_2}
\end{eqnarray}
where we set $ \bmu^k = \lambda(\bs^k-\bs)$ for all $ k $.
	
The following theorem, which is adapted from \cite[Theorem~3]{EcksteinBertsekas:1992dr}, provides a general convergence result for the augmented Lagrangian scheme \eqref{AugLag_X_1}-\eqref{AugLag_X_2} when the minimization in \eqref{AugLag_X_1} is performed inexactly.
	
\begin{theorem}\label{theorem:Bregman_Convergence_Eckstein}
Let $ K(\bx) $ be a closed convex function, and let $ K(\bx)+\Vert M\,\bx\Vert^2 $ be strictly convex. Let $ \bmu^0\in\Re^{n_s} $ and $ \bx^0\in\Re^{n_x}$ be arbitrary and let $ \lambda>0 $. Suppose that
\begin{itemize}
    \setlength\itemsep{6pt}
	\item[(i)] $ \left\Vert \bx^{k+1} - \argmin\limits_{\bx} \left( K(\bx)- \left\langle\bmu^k,\,M\,\bx\right\rangle + \dfrac{\lambda}{2}\Vert M\,\bx-\bs\Vert^2 \right) \right\Vert < \nu_k$,
	\item[(ii)] $ \bmu^{k+1} = \bmu^k + \lambda\left(\bs - M\,\bx^{k+1}\right) $,
\end{itemize}
where $ \nu_k\geq0 $ and $ \sum_{k=0}^{\infty}\nu_k < +\infty $.
If there exists a saddle point $ \left(\hat{\bx},\hat{\bmu}\right) $ of the Lagrangian function
$$ \mL(\bx,\bmu) = K(\bx) - \left\langle \bmu,\,M\,\bx - \bs\right\rangle, $$
then $ \bx^k\rightarrow\hat{\bx} $ and $ \bmu^k \rightarrow \hat{\bmu}$. If no such saddle point exists, then at least one of the sequences $ \{\bx^k\} $ and $ \{\bmu^k\} $ is unbounded.
\end{theorem}
\begin{proof}
For each $ k $, let $ \bar{\bx}^k $ be the unique solution to the minimization problem in (i) (the uniqueness comes from the strict convexity of $ K(\bx)+\Vert M\,\bx\Vert^2 $).
Since $ \bar{\bx}^k $ is a stationary point, it satisfies the necessary condition
\begin{equation}\label{ConvThm_stationarity_1}
    \bzero \in \partial K(\bar{\bx}^k) - M^\top \bmu^k + \lambda\,M^\top \left(M\,\bar{\bx}^k - \bs\right).
\end{equation}
By defining $ \widetilde{\bmu}^k = \bmu^k - \lambda\left(M\,\bar{\bx}^k-\bs\right) $, condition \eqref{ConvThm_stationarity_1} can be written as
$$ M^\top \widetilde{\bmu}^k\in\partial K(\bar{\bx}^k) $$
which, by \eqref{conjugate_subdifferential}, is equivalent to
$$ \bar{\bx}^k\in \partial K^*(M^\top \widetilde{\bmu}^k).$$
Therefore,
\begin{equation}\label{ConvThm_stationarity_2}
    M\,\bar{\bx}^k - \bs \in \Psi(\widetilde{\bmu}^k),
\end{equation}
where $ \Psi(\widetilde{\bmu}^k) \equiv M\,\partial K^*(M^\top \widetilde{\bmu}^k)-\bs $.
From the definition of $ \widetilde{\bmu}^k $ and \eqref{ConvThm_stationarity_2} it follows that
$$ \widetilde{\bmu}^k = \bmu^k - \lambda\left(M\,\bar{\bx}^k-\bs\right) \in \bmu^k - \lambda\,\Psi(\widetilde{\bmu}^k),$$
that is
\begin{equation}
	\bmu^k \in \widetilde{\bmu}^k + \lambda\,\Psi(\widetilde{\bmu}^k) =
	\left(I+\lambda\,\Psi\right)(\widetilde{\bmu}^k).\label{ConvThm_mutilde1}
\end{equation}
Observe that $ \Psi(\bmu) = \partial\left(K^*(M^\top  {\bmu}) - \left\langle \bs,\bmu\right\rangle \right) $, i.e., it is the subdifferential of a closed convex function. From \cite[Corollary~31.5.2]{Rockafellar:1970convex} we have that $ \Psi $ is a maximal monotone operator. Thus, by \cite[Corollary~2.2]{EcksteinBertsekas:1992dr}, for any $ c>0 $ the operator $ J_{c\Psi} \equiv \left(I+c\,\Psi\right)^{-1} $ is single valued and has full domain. By \eqref{ConvThm_mutilde1}, we have
\begin{equation*}
	\widetilde{\bmu}^k = \left(I+\lambda\,\Psi\right)^{-1}(\bmu^k) = J_{\lambda\Psi}(\bmu^k).
\end{equation*}
Thus, by hypothesis (i) we get
\begin{equation} \label{ConvThm_error_mu}
	\left\|\bmu^{k+1} - \left(I+\lambda\,\Psi\right)^{-1}(\bmu^k) \right\| =
     \left\|\bmu^{k+1} - \widetilde{\bmu}^k \right\| \leq \lambda\|M\|\left\|\bx^{k+1} - \bar{\bx}^k \right\| < \lambda\|M\|\nu_k \equiv \beta_k,
\end{equation}
with  $ \sum_{k=0}^{\infty}\beta_k < +\infty $. By \cite[Theorem~3]{EcksteinBertsekas:1992dr} we have that the sequence $ \{\bmu^k\} $ satisfies one of the two following conditions:
\begin{itemize}
	\item[1)] if $\Psi $ has a zero, i.e., there exists a vector $ \hat{\bmu}$ such that
	$$\Psi(\hat{\bmu}) = M\,\partial K^*(M^\top \hat{\bmu})-\bs = \bzero,$$
	then $ \bmu^k\rightarrow \hat{\bmu}$;
	\item[2)] if $ \Psi $ has no zeros, then the sequence is unbounded.
\end{itemize}
	
Now we prove that in case 1) the sequence $\{ \bx^k \}$ converges to a point $ \hat{\bx} $. To this aim, we consider the minimization problem in (i). By defining $ Z(\bx)\equiv K(\bx) + \dfrac{\lambda}{2}\Vert M\,\bx-\bs\Vert^2$, which is a strictly convex function by hypothesis, we can write the stationarity condition for $ \bar{\bx}^k $ as
$$\bzero \in \partial Z(\bar{\bx}^k) - M^\top \bmu^k,$$
or equivalently as
$$\bar{\bx}^k\in\partial Z^*(M^\top {\bmu}^k).$$

The strict convexity of $ Z $ implies that $ Z^*$ is a continuously differentiable function and hence
$$\bar{\bx}^k = \nabla Z^*(M^\top {\bmu}^k),$$
which implies
$$\bar{\bx}^k \rightarrow \hat{\bx}\equiv\nabla Z^*(M^\top \hat{\bmu}).$$
This, together with $ \left\|\bx^{k+1} - \bar{\bx}^k \right\|<\nu_k\rightarrow 0 $, yields $ \bx^k\rightarrow \hat{\bx} $.
	
Now we show that the pair $\left(\hat{\bx},\hat{\bmu}\right)$ is a saddle point of the Lagrangian function $ \mL(\bx,\bmu) $, i.e., it satisfies
	\begin{itemize}
		\item[a)] $ \bzero\in \partial_\bx \mL(\hat{\bx},\hat{\bmu}) = \partial K(\hat{\bx})-M^\top \hat{\bmu}$ or, equivalently, $M^\top \hat{\bmu} \in \partial K(\hat{\bx})$;
		\item[b)] $ \bzero = \nabla_{\bmu} \mL(\hat{\bx},\hat{\bmu}) = M\,\hat{\bx}-\bs $. 
	\end{itemize}
The proof of b) follows by noting that $ M\,\bx^{k+1}-\bs = \dfrac{1}{\lambda}(\bmu^k - \bmu^{k+1})\rightarrow\bzero $.
In order to prove~a) we observe that $ \bar{\bx}^k\rightarrow\hat{\bx}$ implies $\widetilde{\bmu}^k\rightarrow \hat{\bmu}$; moreover, $  M^\top \widetilde{\bmu}^k\in\partial K(\bar{\bx}^k) $.
The thesis comes from the limit property of maximal monotone operators \cite{Brezis:1973} applied to $ \partial K$.
\end{proof}

\begin{remark}\label{remark:theorem_Bregman_Convergence_Eckstein}
Because of the equivalence between \eqref{SimpBreg_X_1}-\eqref{SimpBreg_X_2} and \eqref{AugLag_X_1}-\eqref{AugLag_X_2}, the previous theorem implies that if $ \bmu^k\rightarrow\hat{\bmu} $, then the sequence $ \{\bs^k\}$ generated in \eqref{SimpBreg_X_1}-\eqref{SimpBreg_X_2} converges to $ \hat{\bs} = \frac{1}{\lambda}\hat{\bmu}+\bs$.
\end{remark}

%\section{Measures of optimality for the Bregman subproblems\label{sec:optmeasures}}
\section{Subspace acceleration for the split Bregman subproblems\label{sec:acceleration}}
%We aim at introducing a subspace-acceleration technique into the Bregman scheme.
Let us introduce, for each $ \bx\in\Re^{n_x} $, the sets
\begin{eqnarray}
    &\mA_+(\bx) = \{ i : x_i > 0 \},\qquad &\mA_-(\bx) = \{ i : x_i < 0 \},\nonumber \\
    &\mA_0(\bx) = \{ i : x_i = 0 \},\qquad &\mA_\pm(\bx) = \mA_+(\bx) \cup \mA_-(\bx).\nonumber 
\end{eqnarray}
This partitioning of the variables has been used in \cite{Solntsev:2015OMS,Robinson:2017FaRSA} to extend some ideas developed in the context of active-set methods for bound-constrained optimization \cite{Friedlander:1994b,Dostal:1997,Dostal:2005} to the case of $\ell_1$-regularized optimization. In the case of bound-constrained quadratic problems, suitable measures of optimality with respect to the \textit{active} variables (i.e., the variables that are on their bounds) and the \textit{free} variables (i.e., the variables that are not active) are used to establish whether the set of active variable is ``promising''. If this is the case, then a restricted version of the problem, obtained by fixing the active variables to their values, is solved with high accuracy. This results in very efficient algorithms in practice, able to outperform standard gradient projection schemes \cite{Robinson:2015,diSerafinoToraldoViolaBarlow:2018}. The extension of this strategy to the case of $\ell_1$-regularized optimization comes from the observation that zero and nonzero variables can play the role of active and free variables, respectively.

The results contained in this Section require a further assumption on the function $ f(\bu) $ in~\eqref{TVl1reg_problem}.
\begin{assumption}\label{assumption:Lipschitz}
	The gradient of $f$ is Lipschitz continuous with constant $L$ over $\Re^n$, i.e., for all $ \bu_1,\bu_2\in\Re^n $
	\begin{equation*}
	\|\nabla f(\bu_1) - \nabla f(\bu_2)\| \leq L \, \|\bu_1-\bu_2\|.
	\end{equation*}
\end{assumption}
Note that $ F(\bx) $ defined in \eqref{const_l1reg_prob_X_definitions} has Lipschitz continuous gradient with the same constant $L$.

In order to ease the description of our acceleration strategy, we reformulate the minimization problem in \eqref{SimpBreg_X_1} as follows:
\begin{equation}
\bx^{k+1} = \argmin\limits_{\bx} H^k(\bx) \equiv G^k(\bx) + \sum_{i=1}^{n_x}\delta_i |x_i|, \label{SimpBreg_X_subp}
\end{equation}
where
$$
     G^k(\bx) = F(\bx) + \dfrac{\lambda}{2}\Vert M\,\bx-\bs^k\Vert^2.
$$
In this way we separate the smooth part of the objective function from the $\ell_1$ regularization term.
Recall that a point $ \bx\in\Re^{n_x} $ is a solution to \eqref{SimpBreg_X_subp} if and only if it satisfies the stationarity condition $\bzero\in\partial H^k(\bx) $, i.e.,
\begin{equation}\label{SimpBreg_X_optimality}
\def\arraystretch{1.1}
\left\lbrace \begin{array}{ll}
\displaystyle \nabla_i G^k(\bx) + \delta_i = 0, & \mbox{if } i \in \mA_+(\bx),\\
\displaystyle \nabla_i G^k(\bx) - \delta_i = 0, & \mbox{if } i \in \mA_-(\bx),\\
\left\vert \nabla_i G^k(\bx) \right\vert \leq \delta_i, & \mbox{otherwise}.
\end{array} \right.
\end{equation}

Consider the pair $\left(\hat{\bx},\hat{\bs}\right)$ defined in Theorem~\ref{theorem:Bregman_Convergence_Eckstein} and in Remark~\ref{remark:theorem_Bregman_Convergence_Eckstein}. Let us define the scalars
$$
     \theta_1 = \dfrac{1}{2} \min\limits_{i\in\mA_\pm(\hat{\bx})} \left\vert \hat{x}_i\right\vert
     \quad \mbox{and} \quad
     \theta_2 = \dfrac{1}{2} \min\limits_{i\in\mA_0(\hat{\bx})} \left(\delta_i - \left\vert\nabla_i \hat{G}(\hat{\bx})\right\vert\right),
$$
where
$$\hat{G}(\bx) = F(\bx) + \dfrac{\lambda}{2}\Vert M\,\bx-\hat{\bs}\Vert^2.$$

We make the following assumptions, which imply that $ \theta_1,\theta_2>0 $.
\begin{assumption}\label{assumption:nonnull}
    The solution $\hat{\bx}$ to problem \eqref{const_l1reg_prob_unk_X} satisfies $\hat{\bx}\neq\bzero$.
\end{assumption}
\begin{assumption}\label{assumption:nondegeneracy}
	The solution $(\hat{\bx},\,\hat{\bs})$ to problem \eqref{const_l1reg_prob_unk_X} is nondegenerate, i.e.
	$$ \min\limits_{i\in\mA_0(\hat{\bx})} \left(\delta_i - \left\vert\nabla_i \hat{G}(\hat{\bx})\right\vert\right) > 0. $$
\end{assumption}

From Assumption~\ref{assumption:Lipschitz} and the definition of $\hat{G}(\bx)$ we have that $ \nabla \hat{G}(\bx) $ is Lipschitz continuous. Indeed, a Lipschitz constant for $ \nabla \hat{G}(\bx) $ is
	$$ \hat{L} = L + \lambda\,\|M\|^2. $$
Since, for any $\bx\in\Re^{n_x}$ and $ k\in\N $,
$$  \nabla G^k(\bx) = \nabla F(\bx) + \lambda\,M^\top (M\,\bx-\bs^k)\quad\mbox{and}\quad \nabla \hat{G}(\bx) = \nabla F(\bx) + \lambda\,M^\top (M\,\bx-\hat{\bs}), $$
we have that for any $ \by,\bz\in\Re^{n_x} $
$$ \left\Vert \nabla G^k(\by) - \nabla G^k(\bz)  \right\Vert = \left\Vert \nabla \hat{G}(\by) - \nabla \hat{G}(\bz)  \right\Vert \leq \hat{L} \|\by-\bz\|. $$
i.e., $ \hat{L} $ is also a Lipschitz constant for $ \nabla G^k(\bx) $.

\noindent
The following lemma shows that when $ \bx^k $ is sufficiently close to $ \hat{\bx} $, then some entries of $ \bx^k $ and $ \hat{\bx} $ have the same sign
(see~\cite[Lemma~3.1]{Robinson:2018OMS_farsa}).

\begin{lemma}\label{Lemma_same_sign}
	If $ \left\|\bx^k - \hat{\bx} \right\|\leq \dfrac{\theta_1}{2}$ then 
	$$ \sign(x^k_i) = \sign(\hat{x}_i), \qquad \forall i\in\mA_\pm(\hat{\bx}) \cup\left(\mA_0(\hat{\bx}) \cap \mA_0(\bx^k)\right). $$
\end{lemma}

We recall that $ \Re^{n_x} $ can be splitted into $ 2^{n_x} $ orthants, and introduce the following definition.
\begin{definition}
	Given any $n_x$-ple $ \sigma\in\{-1,1\}^{n_x} $, the orthant associated with $ \sigma $ is defined as
	$$ \Omega_\sigma = \left\lbrace \bx\in\Re^{n_x} : \; (x_i\geq0 \mbox{ if }\sigma_i=1) \;\wedge\; (x_i\leq0 \mbox{ if }\sigma_i=-1)\right\rbrace. $$ 
\end{definition}

\begin{remark}
\label{remark_same_sign}
Lemma~\ref{Lemma_same_sign} suggests that when the current iterate $ \bx^k $ is close to the solution $\hat{\bx}$, the nonzero entries of $ \bx^k $ have the same sign as the corresponding entries of the solution $ \hat{\bx} $, i.e., $ \bx^k $ and $ \hat{\bx} $ lie in the same orthant of $ \Re^{n_x} $. Therefore one could think of restricting the current subproblem~\eqref{SimpBreg_X_subp} to the orthant containing $ \bx^k $.
The restriction of $ H^k(\bx) $ to an orthant $ \Omega_\sigma $ has the form
\begin{equation*}
H^k_{| \Omega_\sigma}(\bx) = G^k_{| \Omega_\sigma}(\bx) + \left\langle\bnu_\sigma,\,\bx\right\rangle,
\end{equation*}
where we set for all $ i $
\begin{equation*}
[\bnu_\sigma]_i = \left\lbrace \begin{array}{rl}
 \delta_i, & \mbox{if } \sigma_i = \phantom{+}1,\\
-\delta_i, &  \mbox{if } \sigma_i = -1.
\end{array} \right.
\end{equation*}

\noindent Since $ H^k_{| \Omega_\sigma}(\bx) $ is a smooth function, if we knew that the current orthant contained the solution to \eqref{SimpBreg_X_subp}, then we could choose to solve the subproblem with high accuracy by using techniques suited for smooth bound-constrained optimization problems. Similar ideas have been exploited in the solution of unconstrained $\ell_1$-regularized nonlinear problems, giving rise to the family of the so-called ``orthant-based algorithms'' \cite{Byrd:2012MathProg,Keskar:2016OBA}.
\end{remark}

\smallskip
We aim at introducing subspace-acceleration steps into the Bregman framework. This means that, at suitable Bregman iterations, we want to replace the minimization of $ H^k$ with the minimization of its restriction to the orthant face determined by $ \mA_0(\bx^k) $, i.e., the set
\begin{equation}\label{equation:orthant_face_xk}
    \left\lbrace \by\in\Re^{n_x} \;:\; \left(y_i = 0,\;i\in\mA_0(\bx^k)\right) \;\wedge\; \left(\sign(y_i)=\sign(x^k_i),\;i\in\mA_\pm(\bx^k)\right) \right\rbrace.
\end{equation}
When $ \mA_0(\bx^k) $ is large, this could result in a significant reduction of the computational cost of determining the next iterate.

Recall that the optimality of a given point $\bx$ with respect to problem \eqref{SimpBreg_X_subp} can be measured in terms of the minimum norm subgradient of $H^k$ at a given point $\bx$, i.e., the vector $\bg^k(\bx)$ defined componentwise as 
$$
[\bg^k(\bx)]_i = \left\{ \begin{array}{ll}
\nabla_i G^k(\bx) + \delta_i, & \mbox{if } i \in \mA_+(\bx) \mbox{ or } (i \in \mA_0(\bx) \mbox{ and } \nabla_i G^k(\bx) + \delta_i  < 0), \\
\nabla_i G^k(\bx) - \delta_i,  & \mbox{if } i \in \mA_-(\bx) \mbox{ or } (i \in \mA_0(\bx) \mbox{ and } \nabla_i G^k(\bx) - \delta_i  > 0),\\
0, & \mbox{otherwise}.
\end{array} \right.
$$
By following \cite{Solntsev:2015OMS,Robinson:2017FaRSA}, we split $\bg^k(\bx)$ into the vectors $ \bbeta^k(\bx) $ and $ \bphi^k(\bx) $, which measure the optimality of $ \bx $ with respect to the zero and nonzero variables respectively. The two vectors are defined componentwise as
\begin{equation}\label{Def_beta_phi}
\def\arraystretch{1.1}
\begin{array}{l}
\left[\bbeta^k(\bx)\right]_i = \left\{  \begin{array}{ll}
\nabla_i G^k(\bx) + \delta_i, & \mbox{if } i \in \mA_0(\bx) \mbox{ and } \nabla_i G^k(\bx) + \delta_i  < 0, \\
\nabla_i G^k(\bx) - \delta_i,  & \mbox{if } i \in \mA_0(\bx) \mbox{ and } \nabla_i G^k(\bx) - \delta_i  > 0, \\
0,                                           & \mbox{otherwise},
\end{array} \right. \\
\left[\bphi^k(\bx)\right]_i = \left\{  \begin{array}{ll}
0,                                           & \mbox{if } i \in \mA_0(\bx), \\
\min \{ \nabla_i G^k(\bx) + \delta_i, \max \{ x_i ,  \nabla_i G^k(\bx) - \delta_i \} \},
& \mbox{if } i \in \mA_+(\bx), \\
\max \{ \nabla_i G^k(\bx) - \delta_i, \min \{ x_i ,  \nabla_i G^k(\bx) + \delta_i \} \},
& \mbox{if } i \in \mA_-(\bx).
\end{array} \right.
\end{array}
\end{equation}

\noindent It is straightforward to check that if $ \bbeta^k(\bar{\bx}) = \bzero $ and $ \bphi^k(\bar{\bx})=\bzero $ at any point $ \bar{\bx}\in\Re^{n_x} $,
then $ \bar{\bx} $ is a stationary point for $ H^k(\bx) $. It is worth noting that the vector $ \bphi^k(\bx) $ also takes into account how much nonzero variables can move before becoming zero, i.e., before $\bx$ enters another orthant \cite{Robinson:2017FaRSA}.

Now we can prove a bound on the components of $ \nabla G^k(\bx^k) $ corresponding to indices in $\mA_0(\hat{\bx})$ when $ (\bx^k,\,\bs^k) $ is ``sufficiently close'' to $ (\hat{\bx},\,\hat{\bs}) $. The result extends to the case of Bregman iterations for problem \eqref{const_l1reg_prob_unk_X} a similar result proved in \cite{Robinson:2018OMS_farsa} for the solution of $\ell_1$-regularized unconstrained minimization problems.
\begin{theorem}
	If $ \left\|\bx^k - \hat{\bx} \right\|\leq \min\left\{\dfrac{\theta_1}{2},\,\dfrac{\theta_2}{2 \hat{L}}\right\}$ and $ \left\|\bs^k - \hat{\bs} \right\|\leq  \dfrac{\theta_2}{2\lambda \|M\|}$, then
	\begin{itemize}
		\item[i)] $ \left| \nabla_i G^k(\bx^k)\right| \leq \delta_i-\theta_2, \qquad \forall i \in\mA_0(\hat{\bx}),$
		\item[ii)] $ \bbeta^k(\bx^k)=\bzero $.
	\end{itemize}
\end{theorem}
\begin{proof}
	In order to prove i), let us consider an index $ k $ satisfying the hypotheses. For all $ i $, we can write
%   \begin{eqnarray*}
		\begin{align}
			\left| \big| \nabla_i G^k(\bx^k) \big|- \big| \nabla_i \hat{G}(\hat{\bx})\big| \right| &\leq \left| \nabla_i G^k(\bx^k) - \nabla_i \hat{G}(\hat{\bx}) \right| = \\
			&   = \left| \nabla_i \hat{G}(\bx^k) + [\lambda\,M^\top (\hat{\bs}-\bs^k)]_i - \nabla_i \hat{G}(\hat{\bx}) \right| \leq\\
			&\leq \left\Vert \nabla \hat{G}(\bx^k) + \lambda\,M^\top (\hat{\bs}-\bs^k) - \nabla \hat{G}(\hat{\bx})  \right\Vert \leq\\
			&\leq \hat{L}\,\|\bx^k-\hat{\bx}\| + \lambda\,\|M\|\,\|\bs^k-\hat{\bs}\|\leq \frac{\theta_2}{2}+\frac{\theta_2}{2} = \theta_2.
		\end{align}
%   \end{eqnarray*}
	This implies that
	\begin{equation}\label{Thm_grad_gradstar}
	\left| \nabla_i G^k(\bx^k)\right| \leq \left| \nabla_i \hat{G}(\hat{\bx})\right| + \theta_2
	\end{equation}
	for all $ i $. Recall that $ \delta_i = \tau_1 $ for $ i\leq n $ and $ \delta_i = \tau_2 $ otherwise. Without loss of generality we analyze the case $ i\leq n $;
	the case $i > n$ can be proved in the same way. By defining
	$$ c_1 = \max\limits_{l\in\mA_0(\hat{\bx})\cap\{1,\ldots,n\}} \left| \nabla_l \hat{G}(\hat{\bx})\right|, $$
	we have that
	$$ \theta_2 \leq (\tau_1 - c_1)/2. $$
	Let $ i\in\mA_0(\hat{\bx})\cap\{1,\ldots,n\} $. From \eqref{Thm_grad_gradstar} and the previous inequality we get
	\begin{eqnarray*}
		\left| \nabla_i G^k(\bx^k)\right| &  \leq  & \left| \nabla_i \hat{G}(\hat{\bx})\right| + \theta_2 \leq c_1 + \frac{\tau_1 - c_1}{2} = \\ % = \frac{\tau_1 + c_1}{2} \\
		&    =    & \tau_1 - \frac{\tau_1 - c_1}{2} \leq \tau_1 - \theta_2 = \delta_i - \theta_2.
	\end{eqnarray*}
	This completes the proof of i).
	
	To prove ii), we observe that $ \beta^k_i(\bx^k)$ can be nonzero only for $ i\in\mA_0(\bx^k) $ and that $ \mA_0(\bx^k) $ can be written as
	$$ \mA_0(\bx^k) = \left(\mA_0(\bx^k)\cap \mA_0(\hat{\bx})\right) \cup \left( \mA_0(\bx^k)\cap \mA_\pm(\hat{\bx}) \right). $$
	From Lemma~\ref{Lemma_same_sign} it follows that $ \mA_0(\bx^k)\cap \mA_\pm(\hat{\bx})=\emptyset $. For $ i\in \mA_0(\bx^k)\cap \mA_0(\hat{\bx}) $ we have
	$$ \left| \nabla_i G^k(\bx^k)\right| \leq \delta_i - \theta_2 < \delta_i,$$
	which concludes the proof.
\end{proof}

The previous theorem suggests that when $ (\bx^k,\,\bs^k) $ is in a neighborhood of the solution $ (\hat{\bx},\,\hat{\bs}) $, the only variables that violate the optimality conditions are the nonzero ones.

By Remark~\ref{remark_same_sign}, the orthant containing the solution is identified as the iterates converge to the solution. Therefore, one could think of introducing into the general inexact Bregman framework \eqref{SimpBreg_X_1}-\eqref{SimpBreg_X_2} an automatic criterion to decide whether the solution to \eqref{SimpBreg_X_1} can be searched in the current orthant face by means of a more efficient algorithm. Inspired by similar conditions introduced in the framework of bound-constrained quadratic problems \cite{Friedlander:1994b,Dostal:1997,Robinson:2015,diSerafinoToraldoViolaBarlow:2018}, we propose to perform subspace-acceleration steps whenever
\begin{equation}\label{switch_criterion}
    \left\Vert \bbeta^k(\bx^k) \right\Vert \leq \gamma\,\left\Vert \bphi^k(\bx^k) \right\Vert,
\end{equation}
where $ \gamma > 0 $ is a suitable constant. The idea is based on the observation that when the optimality violation with respect to the zero variables is smaller than the violation with respect to the nonzero ones, restricting the minimization to the latter could be more beneficial.

Moreover, since one could expect that for $ (\bx^k,\,\bs^k) $ ``sufficiently close'' to $ (\hat{\bx},\,\hat{\bs}) $ the minimizer of problem \eqref{SimpBreg_X_subp} lies in the same orthant face as $ \bx^k $, it is possible to replace the minimization of $ H^k(\bx)$ over the orthant face containing $\bx^k$ with the minimization over the affine closure of the orthant face, i.e.,
\begin{equation}\label{equation:aff_closure_orthant_face_xk}
\mF^k = \left\lbrace \by\in\Re^{n_x} :\; y_i = 0,\;i\in\mA_0(\bx^k) \right\rbrace.
\end{equation}
This results in replacing the nonsmooth unconstrained minimization problem \eqref{SimpBreg_X_subp} with the smooth optimization problem
\begin{equation}\label{SimpBreg_X_AccSubp}
    \bz^{k+1} = \argmin\limits_{x_i = 0,\,i\in\mA_0^k} \displaystyle H^k_{| \mF^k}(\bx) \equiv G^k_{| \mF^k}(\bx) + \sum_{i\in\mA_\pm^k}\nu^k_i\,x_i,
\end{equation}
where we set $ \mA_\pm^k \equiv \mA_\pm(\bx^k) $, $ \mA_0^k \equiv \mA_0(\bx^k) $ and for all $ i\in\mA_\pm^k $
\begin{equation*}
\nu^k_i = \left\lbrace \begin{array}{rl}
\delta_i, & \mbox{if } \sign(\bx^k) = {+}1,\\
-\delta_i, &  \mbox{if } \sign(\bx^k) = -1.
\end{array} \right.
\end{equation*}

It is worth noting that, by fixing the zero variables, problem \eqref{SimpBreg_X_AccSubp} can be equivalently rewritten as an unconstrained minimization over $ \Re^{|\mA_\pm^k|} $. Therefore, efficient algorithms for unconstrained smooth optimization can be exploited for its solution. Since criterion \eqref{switch_criterion} does not guarantee that $ \bz^{k+1} $ lies in the same orthant as $ \bx^k $, we select the iterate $ \bx^{k+1} $ by a projected backtracking line search ensuring a sufficient decrease in $H^k$, i.e.,
\begin{equation}\label{eqn:sufficient_decrease}
     H^k(\bx^{k+1}) - H^k(\bx^k) \le \eta \langle \nabla H^k(\bx^k), \, \bx^{k+1} - \bx^k \rangle, 
\end{equation}
where $\eta$ is a small positive constant. Note that the orthogonal projection $\pr (\bz; \bx)$ of a point $ \bz $ onto the orthant face containing $ \bx $ can be easily computed componentwise as
$$
\left[\pr (\bz; \bx)\right]_i =
\left\{  \begin{array}{cl}
\max \{0, z_i\}, & \mbox{if } i \in \mA_+(\bx), \\
\min \{0, z_i\},  & \mbox{if } i \in \mA_-(\bx), \\
0,                    &  \mbox{if } i \in \mA_0(\bx).
\end{array} \right.
$$

The resulting method, which we call \emph{Split Bregman with Subspace Acceleration} (SBSA) is outlined in Algorithm~\ref{alg:SBSA}.

\begin{algorithm}[ht!]
	\small
	\caption{Split Bregman with Subspace Acceleration (SBSA)}
	\label{alg:SBSA}
	\begin{algorithmic}[1]
		\State Choose $\bx^0 = \bzero \in \Re^{n_x}$, $\bs^0 = \bzero \in\Re^{n_s}$, $\lambda>0$, $\gamma > 0$;
		%  $\{ \nu^k \}$ be such that $\nu^k \ge 0$ and $\sum_{k=0}^\infty \nu^k < \infty$
		\State $\bx^1 \approx \argmin_{\bx} H^0(\bx)$; \label{alg:SBSA_first_step_fista}
		\For{ $k = 1, 2, \ldots$ }
		\State $\bs^k = \bs^{k-1} + \, \bs - M \, \bx^k$; \label{alg:SBSA_sk}
		\If{ $\| \bbeta^k(\bx^k) \| \le \gamma  \| \bphi^k(\bx^k) \|$ }
		\State $\bz^{k+1} \approx  \argmin \left\lbrace H^k_{| \mF^k}(\bx) \;:\; x_i = 0,\,i\in\mA_0^k \right\rbrace$;\label{alg:SBSA_subsp_acc}
		\State $\bx^{k+1} = \pr \left( \bx^k + \alpha^k (\bz^{k+1} - \bx^k);\; \bx^k\right)$ with $\alpha^k$ obtained by backtracking line search; \label{alg:SBSA_line_search}
		\If{ $\bx^{k+1}$ not sufficiently accurate}  \Comment{\textsc{safeguard}} \label{alg:SBSA_safeguard_1}
		\State $\bx^{k+1} \approx \argmin_{\bx} H^k(\bx)$; \label{alg:SBSA_fista call_safeguard}
		\EndIf \label{alg:SBSA_safeguard_2}
		\Else
		\State $\bx^{k+1} \approx \argmin_{\bx} H^k(\bx)$; \label{alg:SBSA_fista call_standard}
		\EndIf
		\EndFor
	\end{algorithmic}
\end{algorithm}

The following theorem, which is an adaptation of \cite[Theorem A.3]{Keskar:2016OBA}, shows that the line search procedure at step~\ref{alg:SBSA_line_search} of Algorithm~\ref{alg:SBSA} is well defined.

\begin{theorem}
	The backtracking projected line search in the acceleration phases of SBSA terminates in a finite number of
	iterations.
\end{theorem}
\begin{proof}
	Consider the $k$-th iteration of algorithm SBSA and suppose an acceleration step is taken.
	Let $\bz^{k+1}$ be the point computed at line~\ref{alg:SBSA_subsp_acc} of Algorithm~\ref{alg:SBSA} and
	$\bd^k = \bz^{k+1}-\bx^k$. By construction we have that $\bd^k_i = 0$ for all $ i\in\mA_0(\bx^k) $.
	By following the first part of the proof of \cite[Theorem A.3]{Keskar:2016OBA}, it is easy to show that there exists $ \bar{\alpha}>0 $ such that $ \pr ( \bx^k + \alpha \bd^k; \bx^k) = \bx^k + \alpha \bd^k$ for all $ \alpha\in (0,\,\bar{\alpha}]$,  i.e., $\bx^k + \alpha \bd^k$ lies in the same orthant face as $ \bx^k $. Since $ \bz^{k+1} $ is an approximate minimizer of $H^k_{\mF^k}$ and $H^k_{\mF^k}$ is convex, $ \bd^k $ is a local descent direction for $ H^k_{\mF^k} $ in $\bx^k$. This ensures that in a finite number of steps the backtracking procedure can find a value of $ \alpha $ guaranteeing sufficient decrease for $H^k_{\mF^k}$.
	By observing that $ H^k(\bx)=H^k_{\mF^k}(\bx) $ for each $ \bx $ lying in the same orthant face as $ \bx^k $, we conclude that the value of $ \alpha $ obtained with backtracking guarantees sufficient decrease of $ H^k(\bx)$.
\end{proof}

According to Theorem~\ref{theorem:Bregman_Convergence_Eckstein}, the convergence of the inexact scheme is only guaranteed when the solution of the subproblem in~\eqref{SimpBreg_X_subp} is sufficiently accurate. For this reason, a safeguard has been considered at lines~\ref{alg:SBSA_safeguard_1}-\ref{alg:SBSA_safeguard_2} of Algorithm~\ref{alg:SBSA}.
%to guarantee that the iterate generated by the subspace acceleration is improved if it does not satisfy the required optimality.
This could be inefficient in practice, because the output of the subspace acceleration is likely to be rejected when the iterate is far from the solution. In our implementation of Algorithm~\ref{alg:SBSA} we use a heuristic criterion to decide whether to accept the iterate generated by the subspace-acceleration step (see Section~\ref{sec:results}). We recall that for the exact Bregman scheme applied to problem \eqref{const_l1reg_prob_unk_X} it can be proved that (see \cite[Proposition~3.2]{Osher:2005mms})
\begin{equation}\label{lincon_violation_decrease}
    \| M \bx^{k+1} - \bs \| \leq \| M \bx^k - \bs \|
\end{equation}
for all $k$.
Based on this observation, we decided to accept the iterate produced by lines~\ref{alg:SBSA_subsp_acc}-\ref{alg:SBSA_line_search} of Algorithm~\ref{alg:SBSA} if \eqref{lincon_violation_decrease} is satisfied. Numerical experiments showed the effectiveness of this choice.

\section{Application: multi-period portfolio selection\label{sec:portfolio}}

Portfolio selection is central to financial economics and is the building block of the capital asset pricing model.
It aims to find an optimal allocation of capital among a set of assets by rational financial targets. 
For medium- and long-time horizons, a multi-period investment policy is considered: the investors can change the allocation of the wealth among the assets over time by the end of the investment, taking into account the time evolution of available information. In a multi-period setting the investment period is partitioned into sub-periods, delimited by the rebalancing dates
at which decisions are taken. More precisely, if $m$ is the number of sub-periods and $t_j = 1,\ldots,m+1$ denote the rebalancing dates, then a decision taken at time $t_j$ is kept in the $j$-th sub-period $[t_j, t_{j+1})$ of the investment. 
The optimal portfolio is defined by the vector
$$
  \bu = [\bu_1 ; \, \bu_2 ; \,  \dots ; \,  \bu_m] ,
$$
where $\bu_j \in \Re^{n_a}$ is the portfolio of holdings at the beginning of period $j$ and $n_a$ is the number of assets.

In a  multi-period mean variance  Markowitz framework,  the optimal portfolio is obtained  by simultaneously minimizing
the risk and maximizing the return of the investment.  
A common strategy to estimate the parameters of the Markowitz model is to use historical data as predictive of the future behavior of asset returns. This typically leads to ill-conditioned numerical problems. Different regularization techniques have been suggested with the aim of improving the problem conditioning. In the last years, $\ell_1$ regularization techniques have been considered to obtain sparse solutions in both the single and multi-period cases, with the aim of reducing costs \cite{Brodie:2009,Corsaro:2019coap,Corsaro:2019anorL1}.  Another useful interpretation of the
$\ell_1$ regularization is related to the amount of shorting in the portfolio. From the financial point of view, negative solutions correspond to short sales, which are generally transactions in which an investor sells borrowed securities in anticipation of a price decline.  A suitable tuning of the regularization parameter permits short controlling in the solution \cite{Brodie:2009}.

We focus on the fused lasso portfolio selection model  \cite{Corsaro:2019anorFused}, where an additional $\ell_1$ penalty term on the variation is added to the classical $\ell_1$ model in order to reduce the transaction costs. Indeed, in the multi-period case, the sparsity of the solution reduces the holding costs, but it does not guarantee low transaction costs if the pattern of nonzeros positions completely changes across periods.
The fused lasso term shrinks toward zero the differences of values of the wealth allocated across the assets between two contiguous rebalancing dates, thus encouraging smooth solutions that reduce transactions.

%We recall that $m$ denotes the number of rebalancing dates and $n$
%the number of available assets.
Let $\br_j\in \Re^{n_a}$ and $C_j \in \Re^{n_a\times n_a}$ contain respectively
the expected return vector and the covariance
matrix, assumed to be positive definite,
estimated at time $t_j$, $j = 1,\ldots, m$. The fused lasso portfolio selection model reads:
\begin{equation} \label{eq:M2}
\def\arraystretch{1.4}
\begin{array}{rl}
\min & \displaystyle \sum_{j=1}^{m} \left\langle \bu_j,\, C_j \bu_j\right\rangle  + \tau_1 \|\bu\|_1 
+ \tau_2 \sum_{j=1}^{m-1} \|\bu_{j+1}-\bu_j \|_1 \\
\mathrm{s.t.} & \displaystyle \left\langle \bu_1,\, \bone_{n_a}\right\rangle  = \xi_{ini},  \\
& \displaystyle  \left\langle \bu_j,\, \bone_{n_a}\right\rangle  = \left\langle \bone_{n_a}+\br_{j-1},\, \bu_{j-1}\right\rangle , \quad j = 2,\ldots,m,  \\
& \displaystyle  \left\langle \bone_{n_a}+\br_{m},\,\bu_m\right\rangle  = \xi_{fin}, \\
\end{array}
\end{equation}
where $\tau_1, \; \tau_2>0$,
$\xi_{ini}$ is the initial wealth, $\xi_{fin}$
is the target expected wealth resulting from the overall investment. %, and $\bone_{n_a}$ is the vector of ones of length ${n_a}$.
The first constraint is the budget constraint.
The strategy is assumed to be self-financing, as required by constraints from
$2$ to $m$, where it is established that at the end of each period the wealth is given by
the revaluation of the previous one.
The $(m+1)$-st constraint defines the expected final wealth. Problem (\ref{eq:M2}) can be equivalently formulated as
\begin{equation} \label{eq:M2comp}
\def\arraystretch{1.4}
\begin{array}{rl}
\min               & \displaystyle \left\langle \bu,\,C \bu\right\rangle  + \tau_1 ||\bu||_1 + \tau_2  \|D\, \bu\|_1 \\
\mathrm{s.t.} & \displaystyle  A \bu = \mathbf{b},
\end{array}
\end{equation}

\noindent
where $C  \in \Re^{n \times n }$, with $n = m \cdot n_a$,
is the symmetric positive definite block-diagonal matrix
$$ C = \diag (C_1, C_2, \ldots, C_m), $$
$D \in \Re^{ (n-n_a)\times n}$ is the discrete difference matrix defined by
$$
d_{ij}= \left\{
\begin{array}{rl}
-1, & \mbox{ if }  j=i ,\\
1, & \mbox{ if }   j=i+n_a ,\\
0, & \mathrm{otherwise},
\end{array}
\right .
$$
$A \in  R^{(m+1) \times n}$ can be regarded as a $(m+1) \times m$ 
lower block-bidiagonal matrix, with blocks of dimension $1 \times n_a$, defined by
$$
a_{ij}= \left \{
\begin{array}{cl}
\bone_{n_a}^\top, & i=j , \\
-( \bone_{n_a} +\mathbf{r} _{i-1})^\top,  & j=i+1 , \\
\bzero_{n_a}^\top, & \mbox{otherwise} ,
\end{array} \right.
$$
and $\bb=(\xi_{ini},0,0,...,\xi_{fin})^\top  \in \Re^{m+1}$.

\subsection{Testing environment\label{sec:tests_metrics}}
%\subsection{Test problems and quality metrics \label{sec:tests_metrics}}

The SBSA algorithm has been tested on three real datasets. Two of them
use a universe of investments compiled by Fama and French\footnote{Data available from \url{http://mba.tuck.dartmouth.edu/pages/faculty/ken.french/data_library.html\#BookEquity}}.
Specifically, the FF48 dataset contains monthly returns of 48 portfolios representing different industrial sectors, and the FF100 dataset includes monthly returns of 100 portfolios on the basis of size and book-to-market ratio. Both datasets consist of data ranging from July 1926 to December 2015. We consider a preprocessing procedure that eliminates the elements with the highest volatilities, so that the number of portfolios in FF100 is reduced to 96. 
In our experiments we use data during periods 
of 10, 20 and 30 years with annual rebalancing, i.e., we consider the
periods July 2005 - June 2015, July 1995 - June 2015, and July 1985 - June 2015. The corresponding test problems
are called FF48-10y, FF48-20y and FF48-30y, respectively.
The third dataset, denoted ES50, contains the daily returns of stocks included
in the EURO STOXX 50 Index Europe's leading blue-chip index for the Eurozone. 
The index covers the 50 largest companies among the 19 supersectors in terms of free-float
market cap in 11 Eurozone countries. 
%The index covers 50 stocks from 11 Eurozone countries: Austria, Belgium, Finland, France, Germany,
%Ireland, Italy, Luxembourg, the Netherlands, Portugal and Spain.
The dataset contains daily returns for each stock in the index from January 2008 to December 2013. For this
test case we consider both annual ($ m = 6 $ years) and quarterly ($ m = 22 $ quarters) rebalancing.
The corresponding test problems are referred to as ES50-A and ES50-Q, respectively.

Following \cite{Corsaro:2019anorFused}, a rolling window for setting up the model parameters is considered. For each dataset, the length of the rolling windows is fixed in order to build positive definite covariance matrices and ensure statistical significance. Different datasets require different lengths for the rolling windows. FF100 requires ten-year data; for FF48 five years are sufficient;
%two-year data are used for SP500, and
one-year data are used for ES50.

Our portfolio is compared with the benchmark one that is based on the 
strategy where the total amount is equally divided among the assets at each rebalancing date. The portfolio built following this strategy is referred to as the multi-period {\em naive} portfolio and it is 
commonly used as a benchmark by investors because it is a simple rule that reduces risk enough to make a profit.
We assume that the investor has one unit of wealth at the beginning of the planning horizon, i.e., $\xi_{ini}=1$. In order to
compare the optimal portfolio with the naive one, we set the expected final wealth equal to that of the naive one,  i.e., $\xi_{fin}=\xi_{naive}$,  where
$$
\xi_{naive}=\frac{1}{n_a}\left(\ldots\left(\frac{1}{n_a}\left(\frac{\xi_{ini}}{n_a} \langle \mathbf{1}_{n_a} + \mathbf{r}_{1}, \, \mathbf{1}_{n_a} \rangle \right)
\langle \mathbf{1}_{n_a}+\mathbf{r}_{2}, \, \mathbf{1}_{n_a} \rangle \right) \ldots\right) \langle \mathbf{1}_{n_a}+\mathbf{r}_{m}, \, \mathbf{1}_{n_a} \rangle.
$$

Following~\cite{Corsaro:2019anorFused}, we consider some performance metrics that take into account the risk and the cost of the portfolio.
The next metric measures the risk reduction factor of the optimal strategy with respect to the benchmark one:
\begin{equation}\label{ratio}
\mbox{ratio} = \frac{\left\langle\bu_{naive},\,C\,\bu_{naive}\right\rangle}{\left\langle\bu_{opt},\,C \, \bu_{opt}\right\rangle}
\end{equation}
\noindent where $\bu_{naive}$ and $\bu_{opt}$ are the naive portfolio and the optimal one, respectively.

Another metric gives the 
percentage of active positions in portfolio, which is an estimate of the holding costs:
\begin{equation}\label{active}
\mbox{density} = \frac{card\left(\left\{\left| [\bu_{j}]_i  \right| \ge \epsilon_1, i=1,...,n_a, j=1,...,m\right\}\right) \cdot 100}{n} \ \%,
\end{equation}
where $card(S)$ denotes the cardinality of the set $S$.
The threshold $\epsilon_1$ is aimed at avoiding too small wealth allocations, since they make no sense in financial terms.
We note that the density of the naive portfolio is $\mbox{density}_{naive}=100\%$, so we have holding costs in each period for all assets. 
Finally, we use a metric giving information about the total number of variations in weights across periods, which are a measure of the transaction costs:
\begin{equation}\label{trace}
\mT=trace(V^\top  V)
\end{equation}
\noindent where  $V\in \Re^{n_a\times(m-1)}$ with
\begin{equation}
\label{eq:matrix_V}
v_{ij} = \left\{ 
\begin{array}{ll}
1, & \mbox{if } \left\vert [\bu_{j}]_i - [\bu_{j+1}]_i\right\vert \ge \epsilon_2,\\
0, & \mbox{otherwise}.\\
\end{array}\right.
\end{equation}

Note that (\ref{trace}) is a pessimistic estimate of the transaction costs because weights could also change for effect of
revaluation. In order to provide more detailed information about the investment, it is convenient to refer also to $||V||_1$, which is the maximum number of variations over the periods, and to  $||V||_\infty$, which is the maximum number of variations over the assets.

The choice of the regularization parameters $\tau_1$ and $\tau_2$ in \eqref{eq:M2}
plays a key role in obtaining solutions that meet the financial requirements. 
Starting from the numerical results in \cite{Corsaro:2019anorFused}, we selected parameters in $\{10^{-4},10^{-3},10^{-2}\}$, guaranteeing a good tradeoff between the performance metrics and the number of short positions for FF48-20y, FF100-20y and ES50-Q problems.
More in detail, we first set $\tau_1$ as the smallest value producing at most $4\%$ of short positions in the solution and then set $\tau_2$ as the value associated with the maximum ratio as defined in \eqref{ratio}.
We set $\tau_1 = \tau_2 = 10^{-2}$ for FF48-20y, $\tau_1 = 10^{-3}$ and $\tau_2 = 10^{-4}$ for FF100-20y, and $\tau_1 = 10^{-3}$ and $\tau_2 = 10^{-4}$ for ES50-Q. For the tests with different horizon times, we decided to keep the same parameter setting if it provided reasonable portfolios.
However, since the values of the parameters corresponding to FF100-20y produced a number of shorts greater than $4\%$ for FF100-30y, we increased them  as  $\tau_1 = 10^{-2}$ and $\tau_2 = 10^{-3}$.

%As regards the regularization parameters $ \tau_1 $ and $ \tau_2 $ in \eqref{eq:M2}, starting from the numerical results in \cite{Corsaro:2019anorFused} we consider the following settings:
%    \begin{itemize}
%        \item $\tau_1 = \tau_2 = 10^{-2}$ for FF48;
%        \item $\tau_1 = 10^{-3}$ and $\tau_2 = 10^{-4}$ for FF100 with length 10 and 20 years;
 %       \item $\tau_1 = 10^{-2}$ and $\tau_2 = 10^{-3}$ for FF100 with length 30 years;
 %       \item $\tau_1 = 10^{-3}$ and $\tau_2 = 10^{-4}$ ES50.
 %   \end{itemize}
%The parameters are those guaranteeing the best balance between risk, number of shorts and transaction costs.

%\subsection{Implementation details\label{sec:implementation}}
\subsection{Implementation details and numerical results\label{sec:results}}
We developed a MATLAB implementation of Algorithm~\ref{alg:SBSA} specifically suited to take into account that problem \eqref{eq:M2comp} is quadratic. 
% - V: Ho messo prima il criterio sulle iterazioni esterne, e specificato le tol usate
The stopping criterion used for both the standard Bregman iterations and the accelerated ones is based on the violation of the equality constraints, i.e., the execution is halted when
$$ \left\| A \bu^k - \bb \right\| \le tol_B, \quad\mbox{and}\quad \left\| D \bu^k - \bd^k \right\| \le tol_B, $$
with $tol_B=10^{-4}$, which guarantees
a sufficient accuracy in financial terms. A maximum number of Bregman iterations, equal to 10000, is also set. The parameter $\lambda$, penalizing the linear constraint violation in \eqref{SimpBreg_X_1}, is set to $1$.

The inner minimization in the standard Bregman iterations, i.e., for the $\ell_1$-regularized problems at lines \ref{alg:SBSA_first_step_fista}, \ref{alg:SBSA_fista call_safeguard} and~\ref{alg:SBSA_fista call_standard} of Algorithm~\ref{alg:SBSA}, is performed by means of the FISTA algorithm from the FOM package.\footnote{https://sites.google.com/site/fomsolver/} We recall that Theorem~\ref{theorem:Bregman_Convergence_Eckstein} requires the error in the solution of the subproblems to satisfy hypothesis (i). This condition cannot be used in practice, not only because the solution to the subproblem in (i) is unknown, but also because the required tolerance becomes too small after a few steps. However, as noted in \cite{EcksteinBertsekas:1992dr,Rockafellar:1976}, the criterion can be replaced by more practical ones.
%If $F(\bx)$ is a strongly convex function, one can decide to accept the new point if the magnitude of the minimal-norm subgradient is small enough.
We decided to stop the minimization when
$$
  \left\| \bz^{l+1} - \bz^l \right\| \le tol_F,
$$
where $\bz^l$ is the $l$-th FISTA iterate and $ tol_F $ is a fixed tolerance. In our tests we set $tol_F=10^{-5}$ for FF48 and FF100, while for ES50 it is necessary to set $tol_F=10^{-6}$ to ensure convergence of SB within the maximum number of outer iterations. The maximum number of FISTA iterations is set to 5000.
% -V: ma \'e necessaria tutta la descrizione del criterio di arresto di FISTA?
% -D: s\'i, perch\'e altrimenti non si capisce come si controlla l'accuratezza della soluzione del sottoproblema, che ha impatto sulla convergenza delle iterazioni di SB.

Regarding the subspace-acceleration steps (line~\ref{alg:SBSA_subsp_acc} of Algorithm~\ref{alg:SBSA}), since they can be easily reformulated as unconstrained quadratic optimization problems, we use the conjugate gradient (CG) method. In this case the minimization is stopped when
$$\left\| \brho^l \right\| \le \left\| \brho^0 \right\|\, tol_{CG},$$
where $\brho^l$ denotes the residual at the $l$-th CG iteration and $ tol_{CG} $ a fixed tolerance. In the tests we set $ tol_{CG} = 10^{-2}$; we also choose a maximum number of CG steps equal to half the size of the subproblem to be solved. In the sufficient decrease condition~\eqref{eqn:sufficient_decrease} we set $\eta=10^{-1}$.

Concerning criterion~\eqref{switch_criterion} for switching between the standard Bregman iterations and the subspace-acceleration steps, we observed that small values of $\gamma$ tend to penalize the execution of acceleration steps, leading to no improvement in the performance of the algorithm. Thus, in order to favor the use of subspace-acceleration steps, we decided to initialize the parameter $\gamma$ equal to $10$ and to update it during the algorithm with an automatic adaptation strategy similar to that used in~\cite{diSerafinoToraldoViolaBarlow:2018}. In particular, the value of $\gamma$ is reduced by a factor 0.9 when \eqref{switch_criterion} holds, i.e., when subspace-acceleration steps are performed, and is increased by a factor 1.1 otherwise. To warmstart the algorithm, we perform 5 standard Bregman iterations before allowing acceleration.

As regards the safeguard at lines~\ref{alg:SBSA_safeguard_1}-\ref{alg:SBSA_safeguard_2} of Algorithm~\ref{alg:SBSA}, by numerical experiments we found that if $\| M \bx^{k+1} - \bs \| > \| M \bx^k - \bs \|$, then it is generally convenient to accept $\bx^{k+1}$, compute $\bs^{k+1}$ according to 
line~\ref{alg:SBSA_sk} and solve by FISTA the subproblem involving $H^{k+1}$.

In order to assess the performance of SBSA, we compared it with two state-of-the-art methods for the solution of problem \eqref{TVl1reg_problem}:
\begin{itemize}
    \item the split Bregman iteration in \cite[Section~3]{Goldstein:2009SB}, which we denote SB;
    \item the accelerated ADMM algorithm proposed in \cite{Chen:2011jcam}, called AL\_SOP.
\end{itemize}
We note that the SB algorithm is equal to the SBSA algorithm without subspace acceleration. In SB we made the same choices as in SBSA for the solution of the $\ell_1$-regularized subproblems with FISTA and the stopping criteria, to make the effect of the acceleration steps clearer.
Regarding AL\_SOP we observe that, by introducing suitable auxiliary variables, problem \eqref{eq:M2comp} can be equivalently written as
\begin{equation}\label{L1_problem_for_ADMM}
\begin{array}{rl}
\min & \displaystyle \left\langle \bu,\,C \bu\right\rangle + \tau_1\,\|\bv\|_1 + \tau_2\,\|\bd\|_1\\
\mbox{s.t.} & \displaystyle A\bu \qquad = \bb,\\
& \displaystyle \bu - \bv \quad = \bzero,\\
& \displaystyle D\,\bu - \bd = \bzero.
\end{array}
\end{equation}
Given $\by^k = [\bu^k;\bv^k;\bd^k]$, the $(k+1)$-st iteration of the ADMM scheme applied to problem~\eqref{L1_problem_for_ADMM} consists of the minimization of a quadratic function to determine $\bu^{k+1}$ and the application of two soft-thresholding operators to determine $ \bv^{k+1} $ and $ \bd^{k+1} $. By adapting the strategy proposed in \cite{Chen:2011jcam}, we introduced at the end of each iteration an acceleration step over the subspace spanned by $\by^{k+1}-\by^{k}$. The choice of the subspace and the parameter $ \varepsilon $ in the acceleration step was made by following the choice in~\cite[Section~4]{Chen:2011jcam}. In order to make a fair comparison between SBSA and AL\_SOP, we decided to use in AL\_SOP the same stopping criterion as in SBSA, with the additional requirement $ \left\| \bu^k - \bv^k \right\| \le tol_B$. Moreover, at each iteration the solution of the quadratic programming problem for computing $ \bu^{k+1} $ was performed by CG with the same stopping criterion used for the subproblems in SBSA. The maximum number of outer iterations for AL\_SOP was set to 25000.

Finally, we also carried out a comparison with a version of SBSA where the last iterate was forced to be a subspace acceleration. In the following this strategy is denoted SBSA-LSA (LSA: Last Step is an Acceleration).

All the tests were performed with MATLAB R2018b on a 2.5 GHz Intel Core i7-6500U with 12 GB RAM, 4 MB L3 Cache, and Windows 10 Pro (ver.~1909) operating system.

\begin{table}
    \centering
    \caption{Execution times (seconds) and outer iterations of the four algorithms. ``---'' indicates that
    the algorithm does not satisfy the stopping criterion within the maximum number of iterations.\label{table:execution_time}}
    {\small
    \begin{tabular}{|l|rr|rr|rr|rr|}
    	\hline
    	          & \multicolumn{2}{c|}{ SBSA} & \multicolumn{2}{c|}{ SBSA-LSA} & \multicolumn{2}{c|}{ SB} & \multicolumn{2}{c|}{ AL\_SOP} \\ \cline{2-9}
    	Problem   &  time &               outit &  time &                   outit &  time &            outit &  time &                  outit \\ \hline
    	FF48-10y  &  2.61 &                   7 &  2.61 &                       7 &  9.38 &              156 &  2.31 &                   2235 \\
    	FF48-20y  &  6.06 &                  11 &  6.06 &                      11 &  9.29 &               53 & 10.16 &                   4353 \\
    	FF48-30y  &  9.12 &                  14 &  9.12 &                      14 & 93.30 &              693 & 38.47 &                   8889 \\
    	FF100-10y &  6.63 &                  13 &  6.63 &                      13 & 35.81 &              121 &  4.52 &                   1502 \\
    	FF100-20y & 17.16 &                  10 & 17.31 &                      11 & 19.87 &               19 & 19.07 &                   2385 \\
    	FF100-30y & 42.10 &                   9 & 42.10 &                       9 & 46.08 &               21 &   --- &                    --- \\
    	ES50-Q    & 30.80 &                 209 & 30.96 &                     210 & 59.59 &              195 &  8.06 &                   2743 \\
    	ES50-A    &  5.05 &                 304 &  5.05 &                     305 & 14.87 &              269 &  0.87 &                   1377 \\ \hline
    \end{tabular}
    } % end small
\end{table}

\begin{table}
    \centering
    \caption{Comparison among the portfolios computed by the four considered algorithms. The values in brackets correspond to solutions without thresholding. $\mT_{naive}$ denotes the transaction cost for the naive solution.\label{table:solution_quality_comparison}}
    {\small
    \begin{tabular}{|l|rrrrrr|l|}
        \hline
         Problem   & \multicolumn{1}{c}{ratio} & \multicolumn{1}{c}{density} & \multicolumn{1}{c}{\# shorts} & \multicolumn{1}{c}{$\mT$} & \multicolumn{1}{c}{$\| V \|_1$} & \multicolumn{1}{c|}{$\| V \|_\infty$} & \multicolumn{1}{c|}{ $\mT_{naive}$} \\ \hline
        \multicolumn{8}{|c|}{SBSA}                                                                                                                                                                                                      \\ \hline
        FF48-10y  &                      2.32 &              15\%  [19.2\%] &                       0  [0] &                 30  [104] &                         6  [10] &                               8  [11] & { 480}                              \\
        FF48-20y  &                      2.28 &            12.6\%  [14.4\%] &                       0  [0] &                 55  [148] &                        11  [20] &                               7   [8] & { 960}                              \\
        FF48-30y  &                      4.64 &            16.3\%  [17.6\%] &                      29 [29] &                109  [274] &                        14  [30] &                              15  [16] & { 1440}                             \\
        FF100-10y &                      2.94 &            10.5\%  [10.5\%] &                      18 [18] &                 82  [110] &                         7  [10] &                              14  [15] & { 960}                             \\
        FF100-20y &                      9.08 &            14.1\%  [15.7\%] &                      81 [82] &                217  [351] &                        16  [17] &                              34  [37] & { 1920}                             \\
        FF100-30y &                      7.07 &             7.2\%   [8.3\%] &                      51 [51] &                174  [279] &                        16  [20] &                              18  [21] & { 2880}                             \\
        ES50-Q    &                      2.48 &            17.9\%  [29.8\%] &                       0  [0] &                 45  [380] &                        10  [22] &                               9  [32] & { 1100}                              \\
        ES50-A    &                      2.25 &            18.3\%  [32.3\%] &                       0  [0] &                 17  [114] &                         3   [6] &                              10  [27] & { 300}                              \\ \hline
        \multicolumn{8}{|c|}{SBSA-LSA}          \\ \hline
        FF48-10y  &                      2.32 &              15\%  [19.2\%] &                       0  [0] &                  30 [104] &                          6 [10] &                                8 [20] &                                     \\
        FF48-20y  &                      2.28 &            12.6\%  [14.4\%] &                       0  [0] &                  55 [148] &                         11 [20] &                                7 [11] &                               \\
        FF48-30y  &                      4.64 &            16.3\%  [17.6\%] &                      29 [29] &                 109 [274] &                         14 [30] &                               15 [16] &                              \\
        FF100-10y &                      2.94 &            10.5\%  [10.5\%] &                      18 [18] &                  82 [110] &                          7 [10] &                               14 [15] &                                 \\
        FF100-20y &                      9.08 &            14.1\%  [15.7\%] &                      81 [82] &                 217 [349] &                         16 [17] &                               34 [37] &                              \\
        FF100-30y &                      7.07 &             7.2\%   [8.3\%] &                      51 [51] &                 174 [279] &                         16 [20] &                               18 [21] &                                 \\
        ES50-Q    &                      2.48 &            17.5\%  [26.6\%] &                       0  [0] &                  47 [332] &                         10 [22] &                                9 [26] &                              \\
        ES50-A    &                      2.25 &            18.3\%  [28.7\%] &                       0  [0] &                  17 [104] &                          3  [6] &                               10 [25] &                              \\ \hline
        \multicolumn{8}{|c|}{SB}                                                                                                                                                                                                                       \\ \hline
        FF48-10y  &                      2.32 &              15\%  [17.5\%] &                       0  [0] &                  30  [93] &                          6 [10] &                                8 [16] &                                     \\
        FF48-20y  &                      2.28 &            12.6\%  [14.9\%] &                       0  [0] &                  55 [165] &                         11 [20] &                                7 [17] &                                     \\
        FF48-30y  &                      4.64 &            16.3\%  [18.0\%] &                      29 [41] &                 109 [286] &                         14 [30] &                               15 [24] &                                     \\
        FF100-10y &                      2.94 &            10.5\%  [10.6\%] &                      18 [18] &                  82 [112] &                          7 [10] &                               14 [15] &                                 \\
        FF100-20y &                      9.08 &            14.1\%  [15.7\%] &                      81 [89] &                 217 [339] &                         16 [18] &                               34 [40] &                                     \\
        FF100-30y &                      7.07 &             7.2\%   [8.8\%] &                      51 [51] &                 175 [312] &                         16 [20] &                               18 [33] &                             \\
        ES50-Q    &                      2.48 &            15.5\%  [28.3\%] &                       0  [0] &                  48 [355] &                         11 [22] &                               10 [26] &                              \\
        ES50-A    &                      2.27 &            18.3\%  [33.3\%] &                       0  [0] &                  16 [105] &                          3  [6] &                               10 [27] &                              \\ \hline
                \multicolumn{8}{|c|}{AL\_SOP}                                                                                                                                                                              \\ \hline
        FF48-10y  &                      2.32 &               15\%  [100\%] &                     0  [194] &                 31  [480] &                          7 [10] &                                8 [48] & \\
        FF48-20y  &                      2.28 &             12.6\%  [100\%] &                     0  [420] &                 55  [960] &                         11 [20] &                                7 [48] & \\
        FF48-30y  &                      4.64 &             16.4\%  [100\%] &                    29  [699] &                113 [1440] &                         14 [30] &                               15 [48] & \\
        FF100-10y &                      2.91 &             12.9\%  [100\%] &                    18  [488] &                107  [960] &                          9 [10] &                               17 [96] & \\
        FF100-20y &                      8.95 &             17.8\%  [100\%] &                    89  [560] &                307 [1920] &                         17 [20] &                               45 [96] & \\
        FF100-30y &                      7.07 &              7.8\%  [100\%] &                    59 [1465] &                201 [2880] &                         18 [30] &                               18 [96] & \\
        ES50-Q    &                      0.85 &              100\%  [100\%] &                     0  [  0] &                698 [1100] &                         18 [22] &                               50 [50] & \\
        ES50-A    &                      2.00 &             45.3\%  [100\%] &                     0  [124] &                 35  [300] &                          3  [6] &                               21 [50] & \\ \hline
        
    \end{tabular}
    } % small
\end{table}

The results of the tests are summarized in Tables~\ref{table:execution_time} and~\ref{table:solution_quality_comparison}. In Table~\ref{table:execution_time} we report, for each problem and each of the four algorithms, the number of
outer iterations and the execution time in seconds. The number of outer iterations shows that SBSA-LSA performed a further
(final) subspace-acceleration step only for the three problems FF100-20y, ES50-Q and ES50-A, without a practical increase of the execution time.
We see that the SBSA versions of the split Bregman algorithm are able to outperform SB
for all the test problems. The reduction of the total time obtained with SBSA and SBSA-LSA varies between $9 \%$ for FF100-30y and $90 \% $ for FF48-30y.
We note that the cost per iteration of AL\_SOP is far smaller than the one of the other algorithms; however, the proposed accelerated method outperforms AL\_SOP in terms of time on problems
FF48-20y, FF48-30y, FF100-20y and FF100-30y. In particular, for FF100-30y AL\_SOP is not able to converge in 25000 iterations,
corresponding to more that 360 seconds. AL\_SOP requires about the same execution time as SBSA for FF48-10y, while it is much faster
for ES50-Q and ES50-A; however, as we will see later, the quality of the solutions of the ES50 problems computed by AL\_SOP is worse.

In Table~\ref{table:solution_quality_comparison} we report the values of the quality metrics described in Section~\ref{sec:tests_metrics} for the portfolios obtained by using the four algorithms. These metrics are computed before and after thresholding the solution with $ \epsilon_1=\epsilon_2=10^{-4} $ (see~\eqref{active} and~\eqref{eq:matrix_V}). The values before thresholding are in brackets. For each algorithm we report a single value for the ratio, since it is not practically affected by thresholding (we obtained the same results up to the fourth or fifth significant digit). The table shows that the portfolios produced by SBSA, SBSA-LSA and SB are equivalent in financial terms, since the corresponding thresholded solutions produce the same ratios, numbers of short positions, densities and transaction costs.
The densities, transaction costs and numbers of shorts obtained before thresholding give further information on the quality of the optimal solution found by 
the algorithms. The additional subspace-acceleration step performed by SBSA-LSA on the ESQ50 problems allows us to obtain solutions with slightly smaller densities and smaller transaction costs.
Inspection of the non-thresholded solutions of SBSA-LSA and SB shows that in general our subspace-accelerated algorithm is able to compute solutions comparable with those of SB in terms of objective function values. On the other hand, the non-thresholded solutions obtained by SBSA-LSA may have slightly smaller densities or transaction costs.
By looking at the thresholded solutions obtained with AL\_SOP we observe that for the FF100 problems they produce portfolios with slightly poorer qualities since they have a little higher densities, shorts and transaction costs. Regarding the ES50 problems, for which AL\_SOP outperformed SBSA and SBSA-LSA in terms of time, we see that the portfolio computed for ES50-A has a smaller ratio and a much greater density and transaction cost as compared with the other methods, while almost all the metrics concerning the portfolio produced for ES50-Q are worse than
those obtained with the other algorithms. In particular, for ES50-Q the ratio is smaller than 1 and hence the computed portfolio it is not able to satisfy the financial requirements.
%It is, indeed, close to the naive portfolio both in terms of density and transaction costs and it is worse than the naive portfolio in terms of risk (the ratio is smaller than 1).

In our opinion, the results suggest that the proposed split Bregman method with subspace acceleration is not only efficient in terms of computational cost, but is also better than the SB and AL\_SOP methods in enforcing structured sparsity in the solution, especially when no thresholding is applied. This behavior seems to depend on the backtracking projected line search performed at each acceleration step, which allows us to set variables exactly to zero.

\section{Conclusions\label{sec:conclusions}}
A Split Bregman method with Subspace Acceleration (SBSA) has been proposed for sparse data recovery with joint $\ell_1$-type regularizers. The acceleration technique, inspired by orthant-based methods, consists in replacing $\ell_1$-regularized subproblems at certain iterations with smooth unconstrained optimization problems over orthant faces identified by zero variables. These smooth problems can be solved by fast methods. Suitable optimality measures are used to decide whether to perform subspace acceleration. Numerical experiments show that SBSA is effective in solving multi-period portfolio optimization problems and outperforms the Split Bregman method and the Accelerated ADMM algorithm proposed in~\cite{Chen:2011jcam} in terms of computational time and quality of the solution.

Future work will focus on the solution of problems where the function $f$ in~\eqref{TVl1reg_problem} is not quadratic, such as those arising in some classification tasks on fMRI data~\cite{Dohmatob:2014}.
% or in source detection problems in electroencephalography~\cite{Becker:2017}.

\section*{Acknowledgments}
The authors thanks the reviewers for their useful comments, which helped them improve this article.
Marco Viola also thanks Daniel Robinson for useful discussions about orthant-based methods for $\ell_1$-regularized optimization problems.

\bibliographystyle{siam}
\bibliography{Biblio_TVl1}
	
\end{document}